\documentclass[oneside,english]{amsart}
\usepackage[T1]{fontenc}
\usepackage[latin9]{inputenc}
\usepackage{graphicx}
\usepackage{amssymb}

\makeatletter
 \theoremstyle{definition}
 \newtheorem{defn}{Definition}
 \theoremstyle{definition}
  \newtheorem{example}{Example}
  \theoremstyle{plain}
  \newtheorem{thm}{Theorem}
  \theoremstyle{plain}
  \newtheorem{prop}{Proposition}
  \theoremstyle{remark}
  \newtheorem*{rem*}{Remark}
  \theoremstyle{plain}
  \newtheorem{cor}{Corollary}
  \theoremstyle{remark}
  \newtheorem*{acknowledgement*}{Acknowledgement}

\makeatother

\usepackage{babel}

\begin{document}

\title{Diagonal Vectors Of Shifted Young Tableaux}

\address{Department of Mathematics, Massachusetts Institute of Technology,
Cambridge, MA 02139}

\email{dorian@mit.edu}


\author{Dorian Croitoru}

\begin{abstract}
We study vectors formed by entries on the diagonal of standard Young
tableaux of shifted shapes. Such vectors are in bijection with integer
lattice points of certain integral polytopes, which are Minkowski
sums of simplices. We also describe vertices of these polytopes, and
construct corresponding shifted Young tableaux.
\end{abstract}
\maketitle

\section{Shifted Young Diagrams And Tableax}

\begin{defn}
Let $\lambda=(\lambda_{1},\dots,\lambda_{n})$ be a partition with
at most $n$ parts. The \emph{shifted Young diagram of shape $\lambda$
}(or just \emph{$\lambda$-shifted diagram})\emph{ }is the set 

\[
D_{\lambda}=\left\{ (i,\ j)\in\mathbb{R}^{2}\vert\ 1\le j\le n,\ j\le i\le n+\lambda_{j}\right\} .\]
We think of $D_{\lambda}$ as a collection of boxes with $n+1-i+\lambda_{i}$
boxes in row $i$, and such that the leftmost box of the $i^{\mathrm{th}}$
row is also in the $i^{\mathrm{th}}$ column. A \emph{shifted standard
Young tableau shape $\lambda$ }(or just \emph{$\lambda$-shifted
tableau}) is a bijective map $T:D_{\lambda}\rightarrow\left\{ 1,\dots,|D_{\lambda}|\right\} $
which is increasing along rows and down columns, i.e. $T(i,j)<T(i,j+1)$
and $T(i,j)<T(i+1,j)$ ($|D_{\lambda}|={n+1 \choose 2}+\lambda_{1}+\cdots+\lambda_{n}$
is the number of boxes in $D_{\lambda}$). The \emph{diagonal vector
}of such a tableau $T$ is $diag(T)=\left(T(1,1),T(2,2),\dots,T(n,n)\right)$. 
\end{defn}
\begin{example}
The following is a shifted standard Young tableau for $n=4,\ \lambda=(4,2,1,0)$.
Its diagonal vector is $(1,\ 4,\ 7,\ 17)$.

\medskip{}

\includegraphics[scale=0.6]{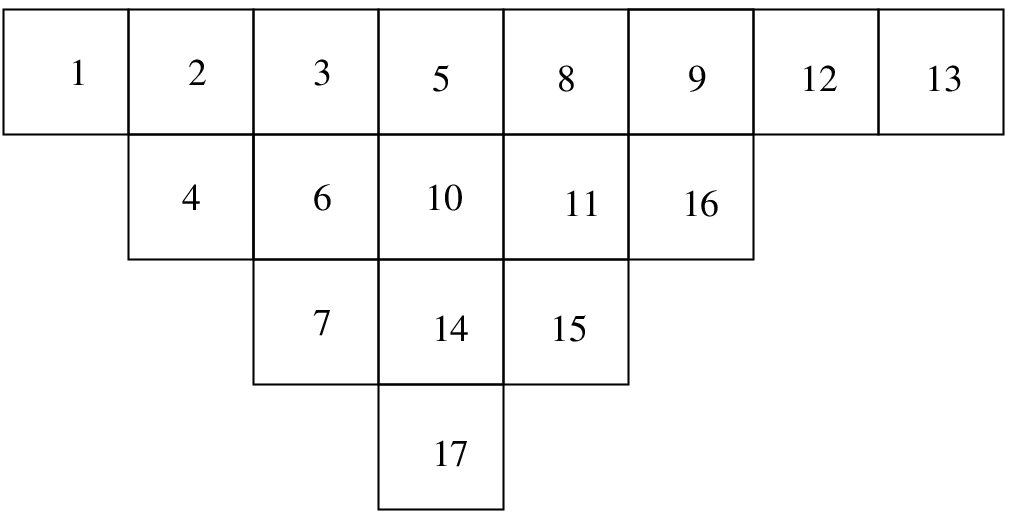}

\medskip{}

\end{example}
We are interested in describing the possible diagonal vectors appearing
in $\lambda$-shifted Young tableaux. The problem was solved in the
case $\lambda=(0,0,\dots,0)$ (the empty partition) by A. Postnikov,
in \cite[Section 15]{key-2}. Specifically, it was shown that diagonal
vectors of the shifted triangular shape $D_{\emptyset}$ are in bijection
with lattice points of the $(n-1)$-dimensional \emph{associahedron}
$\textrm{Ass}_{n-1}$(to be defined in section 2). Moreover, a simple
explicit construction was given for the ''extreme'' diagonal vectors,
i.e. the ones corresponding to the vertices of $\textrm{Ass}_{n-1}$. 

In this article, we aim to generalize Postnikov's results to arbitrary
shifted shapes. Specifically, in section 2 we will prove that diagonal
vectors of shifted $\lambda$-tableaux are in bijection with lattice
points of a certain polytope $\mathbf{P}_{\lambda}$. This polytope
is a Minkowski sum of simplices in $\mathbb{R}^{n}$ and its combinatorial
structure only depends on the length of the partition $\lambda$.
In particular, if the length is $n$, $\mathbf{P}_{\lambda}$ turns
out to be combinatorially equivalent to $\mathrm{Ass}_{n}$. In section
3 we shall give an exlpicit construction of $\lambda$-shifted tableaux
whose diagonal vectors coorespond to the vertices of $\mathbf{P}_{\lambda}$.

\medskip{}

For a non-negative integer vector $\left(a_{1},...,a_{n}\right)$,
let $N_{\lambda}(a_{1},\dots,a_{n})$ be the number of standard $\lambda$-shifted
tableaux $T$ such that $T(i+1,i+1)-T(i,i)-1=a_{i}$ for $i=1,\dots,\ n$,
where we set $T(n+1,n+1)={n+1 \choose 2}+\lambda_{1}+\cdots+\lambda_{n}+1$
. 

\begin{thm}
\label{thm:main}We have the following identity: 

\[
\sum_{a_{1},\dots,a_{n}\ge0}N_{\lambda}(a_{1},\dots,a_{n})\frac{t_{1}^{a_{1}}}{a_{1}!}\cdots\frac{t_{n}^{a_{n}}}{a_{n}!}=\]

\[
=\frac{1}{\prod_{i=1}^{n}(\lambda_{i}+n-i)!}\cdot\prod_{1\le i<j\le n}(t_{i}+\cdots+t_{j-1})\cdot s_{\lambda}(t_{1}+\cdots+t_{n},t_{2}+\cdots+t_{n},\dots,t_{n})\]

where $s_{\lambda}$ denotes the Schur symmetric polynomial associated
to $\lambda$.
\end{thm}
\begin{proof}
Consider a vector $\mathbf{x}=(x_{1}>x_{2}>\dots>x_{n})$. Define
the polytope 

\[
P_{\lambda}(\mathbf{x})=\{(p_{ij})_{(i,j)\in D_{\lambda}}|\ 0\le p_{ij}\ge p_{i(j+1)},\ p_{ij}\ge p_{(i+1)j},\ p_{ii}=x_{i}\}.\]
Thus $P_{\lambda}(\mathbf{x})$ is the section of the order polytope
of shape $D_{\lambda}$ where the values along the main diagonal are
$x_{1},...,\ x_{n}$. If $\lambda=\emptyset$, this polytope is known
as the \emph{Gelfand-Tsetlin polytope}, which has important connections
to finite-dimensional representations of $\mathfrak{gl}_{n}\mathbb{C}$
(see \cite{key-4}).\emph{ }Our proof strategy is to compare two different
formulas for the volume of $P_{\lambda}(\mathbf{x})$, one of which
is more direct and the other is a summation over standard $\lambda$-shifted
Young tableaux. By \cite[Proposition 12]{key-1}, 

\begin{eqnarray}
vol(P_{\lambda}(\mathbf{x})) & = & \frac{1}{\prod_{i=1}^{n}(\lambda_{i}+n-i)!}\cdot\prod_{1\le i<j\le n}(x_{i}-x_{j})\cdot s_{\lambda}(\mathbf{x}).\label{eq:vol1}\end{eqnarray}
On the other hand, there is a natural map $\phi$ from $P_{\lambda}(\mathbf{x})$
(defined except on a set of measure 0), to the set of standard $\lambda$-shifted
Young tableaux, given as follows: Let $\mathbf{p=}(p_{ij})_{(i,\ j)\in D_{\lambda}}\in P_{\lambda}(\mathbf{x})$
be a point with distinct coordinates. Arrange the $p_{ij}$'s in decreasing
order and define the tableau $T=\phi(\mathbf{p})$ by writing $k$
in box $(i,\ j)$ if $p_{ij}$ is the $k^{\mathrm{th}}$ element in
the above list. By the definition of $P_{\lambda}(\mathbf{x})$, it
is clear that $T$ is a standard $\lambda$-shifted Young tableau.
Given a standard $\lambda$-shifted tableau $T$ with diagonal vector
$diag(T)=\{d_{1},\dots,d_{n}\}$, it is easy to see that $\phi^{-1}(T)$
is isomorphic to the set 

\[
\{(y_{i})\in\mathbb{R}^{|T|}\vert\ y_{1}>y_{2}>\dots>y_{|T|}>0,\ y_{d_{i}}=x_{i}\}\]
which is a direct product of (inflated) simplices

\[
\{x_{1}=y_{1}>y_{2}\dots>y_{d_{2}-1}>x_{2}\}\times\cdots\times\{x_{n}=y_{d_{n}}>y_{d_{n}+1}\dots>y_{|T|}>0\}\]
 Therefore, 

\[
vol(\phi^{-1}(T))=\frac{(x_{1}-x_{2})^{a_{1}}}{a_{1}!}\cdot\cdots\cdot\frac{(x_{n-1}-x_{n})^{a_{n-1}}}{a_{n-1}!}\cdot\frac{x_{n}^{a_{n}}}{a_{n}!}.\]
Summing over all $T$, we obtain

\begin{eqnarray*}
vol(P_{\lambda}(\mathbf{x})) & = & \sum_{T}vol(\phi^{-1}(T))\\
 & = & \sum_{a_{1},\dots,a_{n}\ge0}N_{\lambda}(a_{1},\dots,a_{n})\frac{(x_{1}-x_{2})^{a_{1}}}{a_{1}!}\cdot\cdots\cdot\frac{(x_{n-1}-x_{n})^{a_{n-1}}}{a_{n-1}!}\cdot\frac{x_{n}^{a_{n}}}{a_{n}!}.\end{eqnarray*}
Comparing the last formula to (\ref{eq:vol1}), and making the substitutions 

$t_{1}=x_{1}-x_{2},\dots,\ t_{n-1}=x_{n-1}-x_{n},\ t_{n}=x_{n}$,
we obtain the identity in the theorem. 
\end{proof}

\section{Generalized Permutohedra}

In this section we recall the setup from \cite[Section 6]{key-2}.
Let $n\in\mathbb{N}$ and let $e_{1},\dots,\ e_{n}$ denote the standard
basis of $\mathbb{R}^{n}$. For a subset $I\in\{1,2,\dots,n\}$, let
$\Delta_{I}=Conv\{e_{i}\vert\ i\in I\}$, which is an $|I|$-dimensional
simplex. A large class of \emph{generalized permutohedra }(cf. \cite[Section 6]{key-2})
is given by subsets of $\mathbb{R}^{n}$ of the form 

\[
P_{n}^{y}(\{y_{I}\})=\sum_{\emptyset\neq I\subseteq\{1,\dots,n\}}y_{I}\Delta_{I}\]
 i.e. $P_{n}^{y}(\{y_{I}\})$ is the Minkowski sum of the simplices
$\Delta_{I}$ scaled by $y_{I}\ge0$. It's not hard to see that if
$y_{I}=y_{J}$, whenever $|I|=|J|$, then $P_{n}^{y}(\{y_{I}\})$
is the usual permutohedron obtained by taking the convex hull of points
$(x_{1},\dots,\ x_{n})$ such that $x_{1},\dots,\ x_{n}$ is a permutation
of the numbers 

\[
z_{[n]}=\sum_{I\subseteq[n]}y_{I},\ z_{[n-1]}=\sum_{I\subseteq[n-1]}y_{I},\dots,z_{\{1\}}=y_{\{1\}}.\]
Generalized permutohedra have been studied extensively in \cite{key-2}.
One particular example of a generalized permutohedron, the \emph{associahedron}
, is defined as $\mathrm{Ass}_{n}=\sum_{1\le i\le j\le n}\Delta_{[i,\ j]}$.
It is also known as the \emph{Stasheff polytope} and it first appeared
in the work of Stasheff (cf. \cite{key-3}.)

\begin{prop}
\label{pro:newt2}For any subsets $I_{1},\dots,I_{k}\subseteq[n]$,
and any non-negative integers $a_{1},\dots,a_{n}$, the coefficient
of $t_{1}^{a_{1}}\cdots t_{n}^{a_{n}}$ in 

\begin{equation}
\prod_{j=1}^{k}\left(\sum_{i\in I_{j}}t_{i}\right)\label{eq:newt2}\end{equation}
is non-zero if and only if $(a_{1},\dots,a_{n})$ is an integer lattice
point of the polytope $\sum_{j=1}^{k}\Delta_{I_{j}}$. 
\end{prop}
\begin{proof}
It's easy to see that the coefficient of $t_{1}^{a_{1}}\cdots t_{n}^{a_{n}}$
in (\ref{eq:newt2}) is non-zero if and only if $(a_{1},\dots,\ a_{n})$
can be written as a sum of vertices of the simplices $\Delta_{I_{1}},\dots,\ \Delta_{I_{k}}$.
By \cite[Proposition 14.12]{key-2}, this happens if and only if $(a_{1},\dots,\ a_{n})$
is a lattice point of $\sum_{j=1}^{k}\Delta_{I_{j}}$. 
\end{proof}
\begin{prop}
\label{pro:schur1}The coefficient of $t_{1}^{a_{1}}\cdots t_{n}^{a_{n}}$
in $s_{\lambda}(t_{1}+\cdots+t_{n},\ t_{2}+\cdots+t_{n},\dots,\ t_{n})$
is non-zero if and only if $(a_{1},\dots,\ a_{n})$ is a lattice point
of the polytope $\lambda_{1}\Delta_{[1,n]}+\lambda_{2}\Delta_{[2,n]}+\cdots+\lambda_{n}\Delta_{\{n\}}$.
\end{prop}
\begin{proof}
Recall that\begin{eqnarray}
s_{\lambda}(t_{1}+\cdots+t_{n},t_{2}+\cdots+t_{n},\dots,t_{n}) & = & \sum_{T}(t_{1}+\cdots+t_{n})^{w_{1}}\cdots\ t_{n}^{w_{n}},\label{eq:schur1}\end{eqnarray}
 where the sum ranges over all \emph{semi-standard} Young tableaux
$T$ of shape $\lambda$ and weight $\mathbf{w}=(w_{1},\dots,w_{n})$,
i.e. $w_{i}$ is the number of $i$'s appearing in $T$ (see \cite{key-5}).
Let $T$ be a SSYT of shape $\lambda$ and weight $\mathbf{w}$. Then
$w_{1}+\cdots+w_{i}\le\lambda_{1}+\cdots+\lambda_{i},\ \forall i=1\dots n$.
Indeed, if we consider the boxes containing the numbers $1,\ 2,\dots,\ i$
in $T$, there can be no more than $i$ of them in the same column.
Hence the number of such boxes is at most the size of the first $i$
rows of $\lambda$, which is $\lambda_{1}+\cdots+\lambda_{i}$. 

It follows that any monomial $t_{1}^{a_{1}}\cdots t_{n}^{a_{n}}$
appearing in $(t_{1}+\cdots+t_{n})^{w_{1}}\cdots t_{n}^{w_{n}}$ also
appears in $(t_{1}+\cdots+t_{n})^{\lambda_{1}}\cdots t_{n}^{\lambda_{n}}$.
On the other hand, $(t_{1}+\cdots+t_{n})^{\lambda_{1}}\cdots t_{n}^{\lambda_{n}}$does
appear in the right side of (\ref{eq:schur1}) as the term corresponding
to the tableau $T$ with 1's in the first row, 2's in the second row,
etc. Therefore, the coefficient of $t_{1}^{a_{1}}\cdots t_{n}^{a_{n}}$
in $s_{\lambda}(t_{1}+\cdots+t_{n},\ t_{2}+\cdots+t_{n},\dots,\ t_{n})$
is non-zero if and only if it is non-zero in $(t_{1}+\cdots+t_{n})^{\lambda_{1}}\cdots t_{n}^{\lambda_{n}}$,
which by Proposition \ref{pro:newt2}, is non-zero if and only if
$(a_{1},\dots,\ a_{n})$ is a lattice point of $\lambda_{1}\Delta_{[1,n]}+\lambda_{2}\Delta_{[2,n]}+\cdots+\lambda_{n}\Delta_{\{n\}}$.
\end{proof}
\begin{thm}
The number of (distinct) diagonal vectors of $\lambda$-shifted Young
tableaux is equal to the number of lattice points of the polytope 

\[
\mathbf{P}_{\lambda}:=\sum_{1\le i\le j\le n-1}\Delta_{[i,j]}+\lambda_{1}\Delta_{[1,n]}+\lambda_{2}\Delta_{[2,n]}+\cdots+\lambda_{n}\Delta_{\{n\}}.\]

\end{thm}
\begin{proof}
By Theorem \ref{thm:main}, and Propositions \ref{pro:newt2}, \ref{pro:schur1}
it follows that $N_{\lambda}(a_{1},\dots,a_{n})\neq0$ if and only
if $(a_{1},\dots,a_{n})$ is an integer lattice point of the polytope 

\[
\sum_{1\le i\le j\le n-1}\Delta_{[i,j]}+\lambda_{1}\Delta_{[1,n]}+\lambda_{2}\Delta_{[2,n]}+\cdots+\lambda_{n}\Delta_{\{n\}}.\]

\end{proof}
In particular, if $\lambda$ has $n$ parts (i.e. $\lambda_{n}>0$),
we see that $\mathbf{P}_{\lambda}$ is combinatorially equivalent
to $\mathrm{Ass}_{n}$.

\section{Vertices of $\mathbf{P}_{\lambda}$}

In what follows we describe the vertices $\mathbf{P}_{\lambda}$ by
using techniques developed in \cite{key-2}. Given a generalized permutohedron
$P_{n}^{y}(\{y_{I}\})=\sum_{\emptyset\neq I\subseteq\{1,\dots,\ n\}}y_{I}\Delta_{I}$,
assume that its \emph{building set} $B=\{I\subseteq[n]\vert\ y_{I}>0\}$
satisfies the following conditions: 

\begin{enumerate}
\item If $I,\ J\in B$ and $I\cap J\neq\emptyset$, then $I\cup J\in B$.
\item $B$ contains all singletons $\{i\}$, for $i\in[n]$.
\end{enumerate}
A $B$\emph{-forest} is a rooted forest $F$ on the vertex set $[n]$
such that

\begin{enumerate}
\item For any $i$, $\mathrm{desc}(i,F)\in B$ ($\mathrm{desc}(i,F)$ is
the set of descendants of $i$ in $F$).
\item There are no $k\ge2$ distinct incomparable nodes $i_{1},\dots,i_{k}$
in $F$ such that $\bigcup_{j=1}^{k}\mathrm{desc}(i_{j},F)\in B$.
\item $\{\mathrm{desc}(i,F)\vert\ i$- root of $F\}=\{I\in B\vert\ I\mathrm{-maximal}\}$. 
\end{enumerate}
We will need the following result of Postnikov: 

\begin{prop}
\label{pro:Vertices}\cite[Proposition 7.9]{key-2} Vertices of $P_{n}^{y}(\{y_{I}\})$
are in bijection with $B$-forests. More precisely, the vertex $v_{F}=(t_{1},\dots,t_{n})$
of $P_{n}^{y}(\{y_{I}\})$ associated with a $B$-forest $F$ is given
by $t_{i}=\sum_{J\in B:\ i\in J\subseteq\mathrm{desc}(i,F)}y_{J}$,
for $i\in[n]$.
\end{prop}
\begin{rem*}
It's not hard to see that Proposition \ref{pro:Vertices} remains
true even if we allow the building set $B$ not to contain the singletons
$\{i\}$. We will make use of this later on. 
\end{rem*}
The combinatorial structure of $\mathbf{P}_{\lambda}$ clearly depends
only on its building set, i.e. the number of non-zero parts of the
partition $\lambda$. Assume $\lambda_{1},\dots,\lambda_{k}>0,\ \lambda_{k+1}=\cdots=\lambda_{n}=0$,
so that the building set of $\mathbf{P}_{\lambda}$ is \[
B_{k}=\{[i,\ j]\vert\ 1\le i\le j\le n-1\}\cup\{[i,\ n]\vert\ 1\le i\le k\}.\]

We first deal with the case $k=n$. Let $T$ be a plane binary tree
on $n$ nodes. For a node $v$ of $T$, denote by $L_{v},R_{v}$ the
left and right branches at $v$. There is a unique way to label the
nodes of $T$ such that for any node $v$, its label is greater than
all labels in $L_{v}$ and smaller than all labels in $R_{v}$. This
labelling is called the \emph{binary search labelling }of $T$. 

\begin{prop}
\label{pro:Bn forests}\cite[Proposition 8.1]{key-2}The $B_{n}$-forests
are exactly plane binary trees on $n$ nodes with the binary search
labeling. 
\end{prop}
Let $T$ be a $B_{n}$-forest. It's easy to see that $\mathrm{desc}(x,T)$
has form $[a,\ n]$ if and only if the path from the root to $x$
always goes to the right. In this case, $\mathrm{desc}(x,T)=[n-|L_{x}|,n]$
and $n-|L_{x}|$ is maximal when $x$ is the right-most node in $T$,
i.e. $x=n$. It follows that $\{\mathrm{desc}(x,T)\vert\ x\in[n]\}\subseteq B_{k}\subseteq B_{n}\Leftrightarrow|L_{n}|\ge n-k$.
This argument together with Proposition \ref{pro:Bn forests} implies

\begin{prop}
\label{pro:Bk forests}The $B_{k}$-forests are exactly plane binary
trees on $n$ nodes with the binary search labeling and such that
$|L_{n}|\ge n-k$, i.e. such that the (left) subtree of the right-most
node in $T$ has size at least $n-k$. 
\end{prop}
\begin{cor}
The number of vertices of $\mathbf{P}_{\lambda}$ is \[
C_{1}C_{n-1}+C_{2}C_{n-2}+\cdots+C_{k}C_{n-k}\]

where $C_{n}=\frac{1}{n+1}{2n \choose n}$ denotes the $n^{\mathrm{th}}$
Catalan number.
\end{cor}
\begin{proof}
By Propositions \ref{pro:Vertices} and \ref{pro:Bk forests}, the
number of vertices of $\mathbf{P}_{\lambda}$ is equal to the number
of plane binary trees $T$ on $n$ nodes such that left subtree $L$
of the right-most node in $T$ has size at least $n-k$. If $|L|=n-i$,
then there are $C_{n-i}$ ways to choose $L$ and $C_{i}$ ways to
choose the tree $T\backslash L$. Summing over $i=1,\dots,\ k$ yields
the desired formula.
\end{proof}
To describe the vertices of $\mathbf{P}_{\lambda}$, recall that plane
binary trees $T$ on $n$ nodes are in bijective correspondence with
the $C_{n}$ subdivisions of the shifted Young diagram $D_{\emptyset}$
into $n$ rectangles. This can be defined inductively as follows:
Let $i$ be the root of $T$ (in the binary search labeling). Then
draw an $(|L_{i}|+1)\times(|R_{i}|+1)$ rectangle. Then attach the
subdivisions corresponding to the binary trees $L_{i},\ R_{i}$ to
the left and, respectively, bottom of the rectangle. 

For a subdivision $\Xi$ of $D_{\emptyset}$ into $n$ rectangles,
the \emph{$i^{th}$} \emph{rectangle }is the rectangle containing
the $i^{\mathrm{th}}$ diagonal box of $D_{\emptyset}$. If $T$ is
the binary tree corresponding to $\Xi$, then the $i^{\mathrm{th}}$
rectangle of $\Xi$ has size $(|L_{i}|+1)\times(|R_{i}|+1)$. In particular,
$|L_{n}|+1$ is the length of the (bottom-right) vertical strip of
the subdivision $\Xi$. 

\begin{example}
Here is a subdivision of $D_{\emptyset}$ and the corresponding binary
tree with the binary search labeling when $n=4$.

\medskip{}

\includegraphics[scale=0.6]{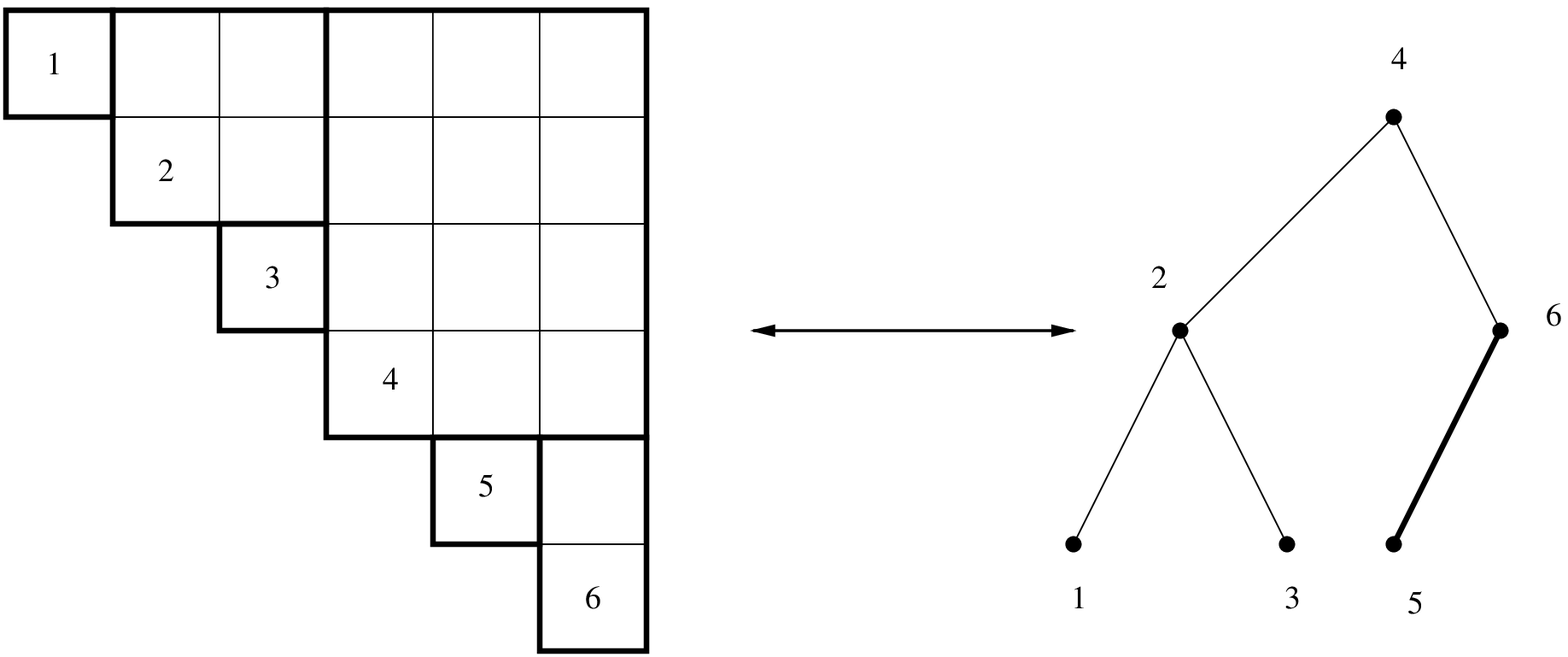}

\smallskip{}

\end{example}
We are finally in a position to prove the main result of this paper. 

\begin{thm}
\label{thm:Vertices_of_Plambda}Vertices of $\mathbf{P}_{\lambda}$
are in bijection with subdivisions of the shifted diagram $D_{\emptyset}$
into $n$ rectangles such that the bottom-right vertical strip of
the subdivision has at least $n-k+1$ boxes. Specifically, let $\Xi$
be such a subdivision. Then we can get a subdivision $\Xi^{*}$ of
$D_{\lambda-\left\langle 1^{k}\right\rangle }$ by merging the rectangles
in $\Xi$ with the rows of the Young diagram of $\lambda-\langle1^{k}\rangle$
that they border. Then the corresponding vertex of $\mathbf{P}_{\lambda}$
is $v_{\Xi}=(t_{1},\dots,t_{n})$, where $t_{i}$ is the number of
boxes in the $i^{\mathrm{th}}$ region of $\Xi^{*}$. 
\end{thm}
\begin{proof}
The first part of the theorem follows from Proposition \ref{pro:Bk forests}
and the discussion preceeding the theorem. To prove the second part,
we use Proposition \ref{pro:Vertices}. Recall that the building set
of $\mathbf{P}_{\lambda}$ is $B_{k}=\{[i,j]\vert\ 1\le i\le j\le n\}\cup\{[i,n]\vert\ 1\le i\le k\}$,
and $\mathbf{P}_{\lambda}=\sum_{[i,j]\in B_{k}}y_{ij}\Delta_{[i,j]}$
where $y_{ij}=1$ if $j\neq1$ and $y_{in}=\lambda_{i}$ . Let $T$
be a $B_{k}$-forest, i.e. a binary tree on $n$ nodes with the binary
search labeling such that $|L_{n}|\ge n-k$ (cf. Proposition \ref{pro:Bk forests}.)
Note that $\mathrm{desc}(i,T)=[i-|L_{i}|,i+|R_{i}|]$. Now Proposition
\ref{pro:Vertices} implies that the correponding vertex $v_{T}=(t_{1},\dots,\ t_{n})$
of $\mathbf{P}_{\lambda}$ is given by 

\begin{eqnarray*}
t_{i} & = & \sum_{J\in B_{k},\ i\in J\subseteq\mathrm{desc}(i,F)}y_{J}=\sum_{[k,l]\in B_{k},\ i-|L_{i}|\le k\le i\le l\le i+|R_{i}|}y_{kl}\\
 & = & (|L_{i}|+1)\cdot|R_{i}|+\sum_{k=i-|L_{i}|}^{i}y_{k(i+|R_{i}|)}.\end{eqnarray*}
If the $i^{\mathrm{th}}$ rectangle of $\Xi$ borders the right edge
of $D_{\emptyset}$ (i.e. $n\in\mathrm{desc}(i,\ T)$), then $t_{i}=(|L_{i}|+1)\cdot|R_{i}|+\sum_{k=i-|L_{i}|}^{i}\lambda_{k}$.
Otherwise, $t_{i}=(|L_{i}|+1)\cdot(|R_{i}|+1)$ . In each case, $t_{i}$
is the number boxes in the $i^{\mathrm{th}}$ region of $\Xi^{*}$.
\end{proof}
\begin{example}
\label{exa:subdiv}Let $n=4,\ \lambda=(4,2,1,0),\ k=3$. The figure
shows how a subdivision $\Xi$ of $D_{\emptyset}$ yields the subdivision
$\Xi^{*}$ of $D_{\lambda-\langle1^{k}\rangle}=D_{(3,1,0)}$. The
corresponding vertex of $\mathbf{P}_{\lambda}$ is given by counting
boxes in the regions of $\Xi^{*}$: $v_{\Xi^{*}}=(1,10,1,2).$ It
follows that there is a (4,2,1,0)-shifted Young tableau $T$ whose
diagonal vector is $\mathrm{diag}(T)=(1,1+1+1,1+1+1+10+1,1+1+1+10+1+2)=(1,3,14,16)$.

\medskip{}

\includegraphics[scale=0.5]{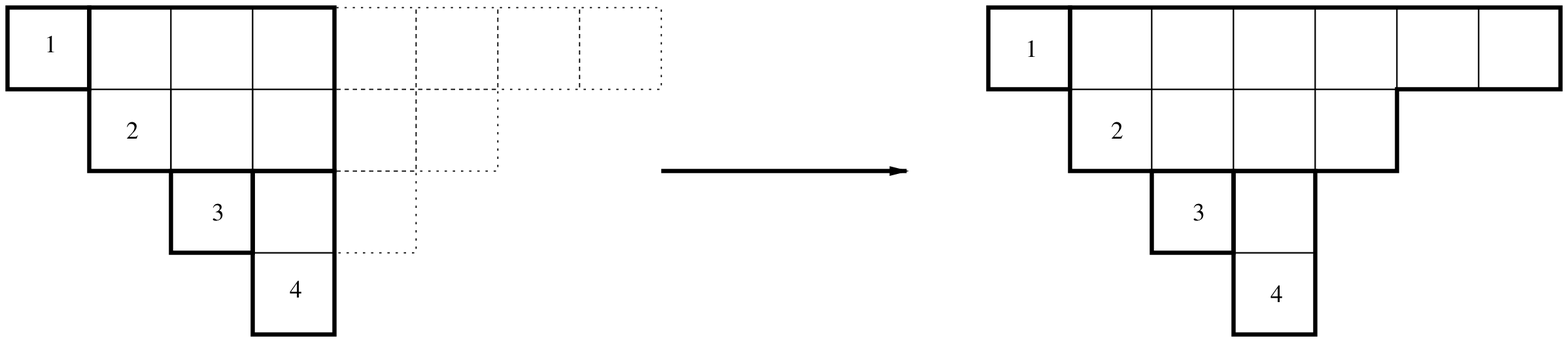}

\medskip{}

On the other hand, one can directly construct $\lambda$-shifted Young
tableaux with diagonal vector $v_{\Xi^{*}}=(c_{1},c_{2},\dots,c_{n})$
by using the subdivision $\Xi^{*}$. Indeed, we know what the diagonal
vector of the tableau $(a_{1},\dots,a_{n})$ should be. Consider again
the subdivision $\Xi^{*}$ of $D_{\lambda-\langle1^{k}\rangle}$.
We can extend the diagram $D_{\lambda-\langle1^{k}\rangle}$ to $D_{\lambda}$
by first adding a box to the left of each row of $D_{\lambda-\langle1^{k}\rangle}$,
and then, by deleting the last $n-k$ boxes in the $n^{\textrm{th}}$
column of $D_{\lambda-\langle1^{k}\rangle}$. Now, we start by putting
$a_{1},\dots,a_{n}$ in the diagonal boxes of $D_{\lambda}$. The
remaining part of $D_{\lambda}$ is divided into $n$ regions by $\Xi^{*}$.
Finaly, for each $i=1,\dots,n$, put the $c_{i}$ numbers $a_{i}+1,\dots,a_{i+1}-1$
in the $i^{\textrm{th}}$ region of $\Xi^{*}$ in a standard way,
i.e. such that entries increase along rows and down columns (as before,
we set $a_{n+1}=|D_{\lambda}|+1$.) In this way we obtain a $\lambda$-shifted
tableau $T$ such that $\textrm{diag}(T)=(a_{1},\dots,\ a_{n})$. 

We illustrate the above procedure for the subdivision in Example \ref{exa:subdiv}. 

\medskip{}

\includegraphics[scale=0.57]{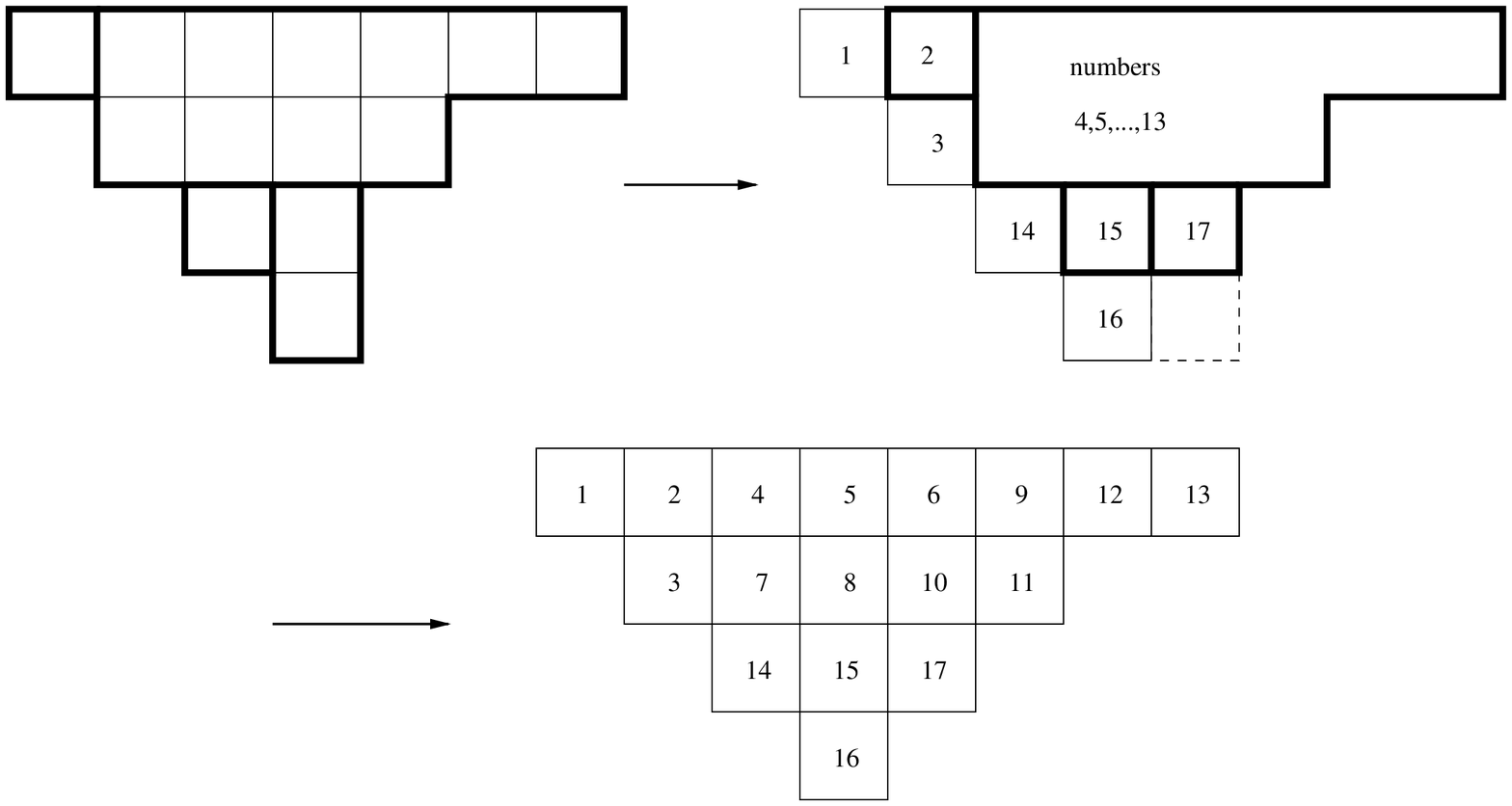}

\medskip{}

\end{example}
\begin{acknowledgement*}
I would like to thank Alexander Postnikov for suggesting the problem
and for all the helpful discussions and ideas regarding the problem.
\end{acknowledgement*}

\end{document}